\documentclass[a4paper]{article}

\usepackage{amssymb}
\usepackage{amsmath}

\newcommand{\rr}{\mathbb{R}}
\newcommand{\ee}{\mathbb{E}}
\newcommand{\ff}{\mathbb{F}}

\newcommand{\nn}{\mathbb{N}}

\newcommand{\pp}{\mathbb{P}}
\newcommand{\qq}{\mathbb{Q}}

\newcommand{\eps}{\varepsilon}
\newcommand{\pspace}{(\Omega,\mathcal{F},\pp)}

\newcommand{\dd}{{\mathrm d} }
\newcommand{\one}{{\bf 1}}

\newcommand{\cf}{\mathcal{F}}

\newcommand{\bq}{\mathbf{Q}}

\newcommand{\ce}{\mathcal{E}}

\newcommand{\cc}{\mathbb{C}}

\newcommand{\bnu}{\bar{\nu}}
\newcommand{\bzeta}{\bar{\zeta}}

\newcommand{\diag}{\mathrm{diag}}

\newcommand{\ii}{\mathrm{i}}

\makeatletter
\def\newrmtheorem#1{\@ifnextchar[{\@rmothm{#1}}{\@rmnthm{#1}}}

\def\@rmnthm#1#2{%
\@ifnextchar[{\@rmxnthm{#1}{#2}}{\@rmynthm{#1}{#2}}}

\def\@rmxnthm#1#2[#3]{\expandafter\@ifdefinable\csname #1\endcsname
{\@definecounter{#1}\@addtoreset{#1}{#3}%
\expandafter\xdef\csname the#1\endcsname{\expandafter\noexpand
  \csname the#3\endcsname \@rmthmcountersep \@rmthmcounter{#1}}%
\global\@namedef{#1}{\@rmthm{#1}{#2}}\global\@namedef{end#1}{\@endrmtheorem}}}

\def\@rmynthm#1#2{\expandafter\@ifdefinable\csname #1\endcsname
{\@definecounter{#1}%
\expandafter\xdef\csname the#1\endcsname{\@rmthmcounter{#1}}%
\global\@namedef{#1}{\@rmthm{#1}{#2}}\global\@namedef{end#1}{\@endrmtheorem}}}

\def\@rmothm#1[#2]#3{\expandafter\@ifdefinable\csname #1\endcsname
  {\global\@namedef{the#1}{\@nameuse{the#2}}%
\global\@namedef{#1}{\@rmthm{#2}{#3}}%
\global\@namedef{end#1}{\@endrmtheorem}}}

\def\@rmthm#1#2{\refstepcounter
    {#1}\@ifnextchar[{\@rmythm{#1}{#2}}{\@rmxthm{#1}{#2}}}

\def\@rmxthm#1#2{\@beginrmtheorem{#2}{\csname the#1\endcsname}\ignorespaces}
\def\@rmythm#1#2[#3]{\@opargbeginrmtheorem{#2}{\csname
       the#1\endcsname}{#3}\ignorespaces}

\def\@rmthmcounter#1{\noexpand\arabic{#1}}
\def\@rmthmcountersep{}
\def\@beginrmtheorem#1#2{\rm \trivlist
      \item[\hskip \labelsep{\bf #1\ #2\thmrmcounterend}]}
\def\@opargbeginrmtheorem#1#2#3{\rm \trivlist
      \item[\hskip \labelsep{\bf #1\ #2\ (#3)\thmrmcounterend}]}
\def\@endrmtheorem{\endtrivlist}
\def\thmrmcounterend{\hskip 0em\relax} 
\def\newrmwntheorem#1#2{\expandafter\@ifdefinable\csname #1\endcsname%
\global\@namedef{#1}{\@rmwnthm{#1}{#2}}%
\global\@namedef{end#1}{\@endrmwntheorem}}

\def\@rmwnthm#1#2{\@ifnextchar[{\@rmwnythm{#1}{#2}}{\@rmwnxthm{#1}{#2}}}
\def\@rmwnxthm#1#2{\@beginrmwntheorem{#2}\ignorespaces}
\def\@rmwnythm#1#2[#3]{\@opargbeginrmwntheorem{#2}{#3}\ignorespaces}
\def\@beginrmwntheorem#1{\rm \trivlist
      \item[\hskip \labelsep{\bf #1\thmrmwncounterend}]}
\def\@opargbeginrmwntheorem#1#2{\rm \trivlist
      \item[\hskip \labelsep{\bf #1\ (#2)\thmrmwncounterend}]}
\def\@endrmwntheorem{\endtrivlist}
\def\thmrmwncounterend{.\hskip 1em\relax}

\def\newsltheorem#1{\@ifnextchar[{\@slothm{#1}}{\@slnthm{#1}}}

\def\@slnthm#1#2{%
\@ifnextchar[{\@slxnthm{#1}{#2}}{\@slynthm{#1}{#2}}}

\def\@slxnthm#1#2[#3]{\expandafter\@ifdefinable\csname #1\endcsname
{\@definecounter{#1}\@addtoreset{#1}{#3}%
\expandafter\xdef\csname the#1\endcsname{\expandafter\noexpand
  \csname the#3\endcsname \@slthmcountersep \@slthmcounter{#1}}%
\global\@namedef{#1}{\@slthm{#1}{#2}}\global\@namedef{end#1}{\@endsltheorem}}}

\def\@slynthm#1#2{\expandafter\@ifdefinable\csname #1\endcsname
{\@definecounter{#1}%
\expandafter\xdef\csname the#1\endcsname{\@slthmcounter{#1}}%
\global\@namedef{#1}{\@slthm{#1}{#2}}\global\@namedef{end#1}{\@endsltheorem}}}

\def\@slothm#1[#2]#3{\expandafter\@ifdefinable\csname #1\endcsname
  {\global\@namedef{the#1}{\@nameuse{the#2}}%
\global\@namedef{#1}{\@slthm{#2}{#3}}%
\global\@namedef{end#1}{\@endsltheorem}}}

\def\@slthm#1#2{\refstepcounter
    {#1}\@ifnextchar[{\@slythm{#1}{#2}}{\@slxthm{#1}{#2}}}

\def\@slxthm#1#2{\@beginsltheorem{#2}{\csname the#1\endcsname}\ignorespaces}
\def\@slythm#1#2[#3]{\@opargbeginsltheorem{#2}{\csname
       the#1\endcsname}{#3}\ignorespaces}

\def\@slthmcounter#1{.\noexpand\arabic{#1}}
\def\@slthmcountersep{}
\def\@beginsltheorem#1#2{\sl \trivlist
      \item[\hskip \labelsep{\bf #1\ #2\thmslcounterend}]}
\def\@opargbeginsltheorem#1#2#3{\sl \trivlist
      \item[\hskip \labelsep{\bf #1\ #2\ (#3)\thmslcounterend}]}
\def\@endsltheorem{\endtrivlist}
\def\thmslcounterend{\hskip 0em\relax}
\def\newslwntheorem#1#2{\expandafter\@ifdefinable\csname #1\endcsname%
\global\@namedef{#1}{\@slwnthm{#1}{#2}}%
\global\@namedef{end#1}{\@endslwntheorem}}

\def\@slwnthm#1#2{\@ifnextchar[{\@slwnythm{#1}{#2}}{\@slwnxthm{#1}{#2}}}
\def\@slwnxthm#1#2{\@beginslwntheorem{#2}\ignorespaces}
\def\@slwnythm#1#2[#3]{\@opargbeginslwntheorem{#2}{#3}\ignorespaces}
\def\@beginslwntheorem#1{\sl \trivlist
      \item[\hskip \labelsep{\bf #1\thmslwncounterend}]}
\def\@opargbeginslwntheorem#1#2{\sl \trivlist
      \item[\hskip \labelsep{\bf #1\ (#2)\thmslwncounterend}]}
\def\@endslwntheorem{\endtrivlist}
\def\thmslwncounterend{.\hskip 1em\relax}

\makeatother

\newsltheorem{theorem}{Theorem}[section]
\newsltheorem{proposition}[theorem]{Proposition}
\newsltheorem{lemma}[theorem]{Lemma}
\newsltheorem{corollary}[theorem]{Corollary}

\newrmtheorem{remark}[theorem]{Remark}

\newenvironment{proof}[1][\proofname]{\par \normalfont \trivlist
 \item[\hskip\labelsep\itshape #1]\ignorespaces
}{%
 \hspace*{\fill}$\Box$ \endtrivlist
}
\newcommand{\proofname}{{\bf Proof}}

\begin{document}

\title{Explicit computations for some Markov modulated counting processes}

\author{Michel Mandjes \and Peter Spreij}



\maketitle

\begin{abstract}
In this paper we present elementary computations for some \emph{Markov modulated} counting processes, also called counting processes with \emph{regime switching}. Regime switching has become an increasingly popular concept in many branches of science. In finance, for instance, one could identify the background process with the `state of the economy', to which asset prices react, or as an identification of the varying default rate of an obligor. 
The key feature of the counting processes in this paper is that their intensity processes are functions of a finite state Markov chain. This kind of processes can be used to model default events of some companies. 

Many quantities of interest in this paper, like conditional characteristic functions, can all be derived from conditional probabilities, which can, in principle, be \emph{analytically} computed. 
We will also study limit results for models with rapid switching, which occur when inflating the intensity matrix of the Markov chain by a factor tending to infinity. 
The paper is largely expository in nature, with a didactic flavor. 
\end{abstract}

\section{Introduction}

In this paper we present some elementary computations concerning some \emph{Markov modulated} (MM) counting processes, denoted $N$, also called counting processes with \emph{regime switching}. Such processes fall into the class of \emph{hybrid models}~\cite{YinZhu} and are in fact Hidden Markov processes~\cite{EAM}. Although in the present paper we restrict ourselves to certain counting processes, it is worth mentioning that owing to its various attractive features, regime switching has become an increasingly popular concept in many branches of science.  In a broad spectrum of application domains it offers a  natural framework for modeling 
situations in which the stochastic process under study reacts to 
an autonomously evolving environment. In finance, for instance, one could identify the background process with the `state of the economy', to which asset prices react, or as an identification of the varying default rate of an obligor. In operations research, in particular in wireless networks, the concept can be used to model the channel conditions that vary in time, and to which users react. In the literature in the latter field there is a sizeable body of work on Markov-modulated queues, see e.g.\ \cite[Ch.\ XI]{ASM} and \cite{NEUTS}, while Markov modulation has been intensively used in insurance and risk theory as well \cite{AA}. In the economics literature, the use of regime switching dates back to at least the late 1980s \cite{HAM}. Various specific models have been considered since then, see for instance \cite{ANG,EMAM,ESIU}. For other direct applications of models with regime switching in finance (hedging of claims, interest rate models, credit risk, application to pension funds) we refer to \cite{chen,JP2008,JP2012,Yin2009,ZhouYin2003} for recent results.

The key feature of the counting processes, commonly denoted $N$, in this paper is that their intensity processes are of the form $\lambda_t=\lambda(X_t,N_t)$, where $X$ is a finite state Markov chain whose jumps with probability one never coincide with the jumps of the counting process. For mathematical convenience we assume without loss of generality that $X$ takes its values in the set of $d$-dimensional basis vectors.

This kind of processes can be used to model default events of some companies. We restrict our treatment to models where the intensity is of a special form, leading to the MM one point process which can be used to model the default event of a single company, its extension to the situation of defaults of various companies and an MM Poisson process, which can be used to model defaults for a large pool of obligors whose individual intensities of default are all the same and small.

The intensities $\lambda_t=\lambda(X_t,N_t)$ that we use will be affine in $X_t$, i.e.\ $\lambda_t=\lambda^\top X_tf(N_t)$ for some $\lambda\in\rr^d$ and some function $f$. It is possible to show that the joint process $(X,N)$ is Markov, in fact it is an affine process after a state transformation. This means that for many quantities of interest, like conditional characteristic functions, one can in principle use the full technical apparatus that has become available for affine process, see \cite{dfs}. However, as these quantities can all be derived from conditional probabilities (our processes are finite, or at most countably, valued), using these techniques is like making a detour since the conditional probabilities can be derived by more straightforward methods. Moreover these conditional probabilities give a \emph{direct} insight into the probabilistic structure of the process and can in principle be \emph{analytically} computed. Therefore, we circumvent the theory of affine processes and focus on direct computation of all conditional probabilities of interest.

We will also study limit results for models with rapid switching, which occur when inflating the intensity matrix of the Markov chain by a factor tending to infinity. Rapid switching between (macro) economic states is unrealistic, but in models for the profit and loss of trading positions, especially in high frequency trading, rapid switching may take place, see~\cite{graziano}.
We will see that the limit processes have intensities that are expectations under the invariant distribution of the chain. This is similar to what happens in the context of Markov modulated Ornstein-Uhlenbeck processes~\cite{HuangMandjesSpreij}, see also \cite{quintet}, whereas comparable results under scaling in the operations research literature can be found in  \cite{BMT} and \cite{BKMT}.

The paper is largely expository in nature, with a didactic flavor. We do not claim novelty of all results below. Rather we emphasize the uniform approach that we follow, using martingale methods, that may also lead to alternative proofs of known results, e.g.\ those concerning transition probabilities by using `$\eps$-arguments' as in \cite{NEUTS}.
The organization of the paper is as follows. In Section~\ref{section:mm} we study Markov modulated model for the total number of defaults when there are $n$ obligors. As a primer, in Section~\ref{section:mm0} we extensively study the Markov modulated model for a single obligor, in particular its distributional properties. 
Then we switch to the more general situation of Section~\ref{section:mmmultiple}, where our approach is inspired by the easier case of the previous section. All results are basically obtained by exploiting the Markovian nature of the joint process $(X,N)$. Section~\ref{section:mmpoisson} gives a few results for the Markov modulated Poisson process. Conditional probabilities of future values of the counting processes, when only its own past can be observed (and not the underlying Markov chain) can be computed using filtering theory, which is the topic of Section~\ref{section:filtering}. In Section~\ref{section:rapid} we obtain the limit results for processes where the Markov chain is rapidly switching.

\section{The MM model for multiple obligors}\label{section:mm}

We assume throughout that a probability space $\pspace$ is given. 
Suppose we have $n$ obligors with default times $\tau^i$ for obligor $i$, $i=0,\ldots,n$. Let $Y^i_t=\one_{\{\tau^i\leq t\}}$, $t\in [0,\infty)$. Here we encounter the canonical set-up for the intensity based approach in credit risk modelling, see \cite[Chapter~12]{filipovic} or \cite[Chapter~6]{br} for further details on probabilistic aspects. 
We postulate for each $i\in\{1,\ldots,n\}$
\begin{equation}\label{eq:yi}
\dd Y^i_t= \lambda_t(1-Y^i_t)\,\dd t + \dd m^i_t,
\end{equation}
for $\lambda_t$ a nonnegative process to be specified, but which is the same for each obligor $i$. Here each $m^i$ is a martingale w.r.t.\ to the filtration, call it $\ff^i$, generated by $Y^i$ and the process $\lambda$.
We impose that the $\tau_i$ are conditionally independent given $\lambda$. Hence, simultaneous defaults occur with probability zero, as the $\tau^i$ have a continuous distribution. By the conditional independence assumption, the $m^i$ are also martingales w.r.t. $\ff=\vee_{i=1}^n\ff^i$. 
The process $\lambda$ is assumed to be predictable w.r.t.\ $\ff$. In all what follows in this section we take $N_t=\sum_{i=1}^nY^i_t$.

\subsection{The MM one point process}\label{section:mm0}

For a better understanding of what follows, we single out the special case $n=1$ and we write $\tau$ instead of $\tau^1$.
There is some advantage in starting with a simpler case that allows for more explicit formulas, is more transparent,  and that at the same time can serve as a warming up for the more general setting.

\subsubsection{The general one point process with intensity}\label{section:inhomo}

Let us consider the basic case, the random variable $\tau$ has an exponential distribution with parameter $\lambda$, and $Y_t=\one_{\{\tau\leq t\}}$, $t\in [0,\infty)$. Then $Y$ has semimartingale decomposition 
\begin{equation}\label{eq:y1}
\dd Y_t = \lambda(1-Y_t)\,\dd t+\dd m_t,
\end{equation}
where $\lambda >0$ and $m$ a martingale w.r.t.\ the filtration generated by the process $Y$. As a matter of fact, the distributional property of $\tau$ is equivalent to the decomposition of $Y$ in \eqref{eq:y1}. Clearly $Y_t$ is a Bernoulli random variable, so $y(t):=\ee Y_t=\pp(Y_t=1)=\pp(\tau\leq t)$. Alternatively, taking expectations, we get the ODE
\[
\dot{y}(t)=\lambda (1-y(t)),
\]
which is, with $y(0)=0$, indeed solved by
\[
y(t)=1-\exp(-\lambda t).
\]
Let $\Lambda^\tau$ be the compensator of $Y$, then
\[
\Lambda^\tau_t=\int_0^t\lambda(1-Y_s)\,\dd s = \int_0^t\lambda\one_{\{s<\tau\}}\,\dd s = \int_0^{t\wedge \tau}\lambda \,\dd s=\lambda (\tau\wedge t).
\]
Note that $Y$ can be considered as $N^\tau$, the at $\tau$ stopped Poisson process   with intensity $\lambda$. The compensator $\Lambda$ of $N$ stopped at $\tau$ indeed yields $\Lambda^\tau$. 
\medskip\\
As a first generalization we change the above setup in the sense that we postulate
\begin{equation}\label{eq:y2}
\dd Y_t = \lambda_t(1-Y_t)\,\dd t+\dd m_t,
\end{equation}
where $\lambda$ is a nonnegative locally integrable Borel function, also known as the (time varying) hazard rate. As above one can show that
\[
y(t)=1-\exp(-\int_0^t\lambda_s\,\dd s).
\]
In a next generalization we suppose that $\lambda$ becomes a random process defined on an auxiliary probability space $(\Omega',\cf',\pp')$. We can look at the product probability space $(\Omega\times\Omega',\cf\otimes\cf',\pp\otimes\pp')$ and redefine in the obvious way $Y$, $\tau$ and $\lambda$ on this product space. The filtration we will use consists of the $\sigma$-algebras $\cf^Y_t\otimes\cf^\lambda_t$.

It is assumed that $\lambda$ is predictable and a.s. locally integrable w.r.t.\ Lebesgue measure. For a given trajectory $\lambda_t=\lambda_t(\omega')$ we define $Y$ on $\pspace$ as in \eqref{eq:y2}. With $\cf^\lambda$ the $\sigma$-algebra generated by the full process $\lambda$, we have that 
\[
\ee[Y_t|\cf^\lambda]=1-\exp(-\int_0^t\lambda_s\,\dd s),
\]
and hence
\[
y(t)=\ee Y_t=1-\ee\exp(-\int_0^t\lambda_s\,\dd s).
\]
Alternatively, one can construct the point process $Y$ as follows. Let $(\Omega, \cf, \qq)$ be a probability space on which is defined a standard Poisson process $Y$ and \emph{independently} of $Y$ the nonnegative predictable process $\lambda$. Put $L_t=\ce(\mu)_t$, the Dol\'eans exponential of the $\qq$-local martingale $\mu$ given by $\mu_t=\int_0^t (\lambda_s\one_{\{Y_{s-}=0\}}-1)\,\dd (Y_s-s)$. Note that $L_0=1$.  Let $\tau_k$ be the consecutive jump times of $Y$, $\tau_0=0$. Note that the differences $\tau_k-\tau_{k-1}$ have a standard exponential distribution under $\qq$. The assertion of the following lemma is a variation on Equation~(4.23) in \cite{br}.
\begin{lemma}
The density process $L$ allows the following explicit expression,
\[
L_t=(\lambda_{\tau_1})^{Y_t}\exp(t-\int_0^{\tau_1\wedge t}\lambda_s\,\dd s)\one_{\{Y_t\leq 1\}}.
\] 
If $\lambda$ is a bounded process, $L$ is a martingale, hence $\ee L_t=L_0=1$.
\end{lemma}

\begin{proof}
By construction, $L$ is a local martingale. For bounded $\lambda$ we have $\ee \int_0^t L_s^2\,\dd s\leq C\exp(2t)$ for some constant $C$, which yields $L$ a square integrable martingale. The given expression for $L_t$ can be verified by an elementary, but slightly tedious computation.
\end{proof}
Under the assumption that $L$ is a martingale (guaranteed for bounded $\lambda$), by Girsanov's theorem, see \cite[Chapter~VI, T3 and T4]{bremaud}, we can define for every $T>0$ a probability $\pp$ on $(\Omega,\cf_T)$ such that 
\[
m_t:=Y_t-t-\langle Y,\mu\rangle_t=Y_t-\int_0^t\lambda_s\one_{\{Y_{s-}=0\}}\,\dd s
\]
is a local martingale under $\pp$. Note that $\pp(Y_T>1)=\ee_\qq\one_{\{Y_T>1\}}L_T=0$. Hence, under $\pp$ we have $\one_{\{Y_s=0\}}=1-Y_s$ and the expression for $m_t$ coincides with \eqref{eq:y2} for $t\leq T$. Note that $L$ cannot be uniformly integrable, since $L_\infty=0$, which follows from $L_{\tau_2}=0$. Hence it is not automatic that one can define a probability $\pp$ on $(\Omega,\cf)$ such that $m$ is a martingale on $[0,\infty)$. Note that the laws under $\pp$ and $\qq$ of $\lambda$ are the same.

\subsubsection{The one point process with MM intensity}\label{section:mm1}

In this section we consider \eqref{eq:y2}, where we specify $\lambda_t$ as a function of a finite state Markov chain $X_t$, i.e.\ $\lambda_t=\lambda(X_t)$. We see that, trivial cases excluded, unlike the constant hazard rate $\lambda$ in \eqref{eq:y1}, we now have a rate that assumes different values according to the states of the Markov chain. We thus have a rate that is subject to \emph{regime switching}, one also says that we have a \emph{Markov modulated} rate. In order to pose a precise mathematical model, we make some conventions. Let $d$ be the size of the state space of the Markov chain $X$. Then w.l.o.g.\ we may assume that $X$ takes its values in the set $\{e_1,\ldots,e_d\}$ of $d$-dimensional standard basis vectors. This implies that any function of $X_t$ can be written as a linear map of $X_t$, in particular $\lambda(X_t)=\lambda^\top X_t$, where on the right hand side $\lambda$ is a vector in $\rr^d_+$. 

Let $Q$ be the transition matrix of $X$, for which we use the convention that $Q_{ij}$ for $i\neq j$ is the intensity of a transition from state $j$ to state $i$. As a consequence the \emph{column sums} of $Q$ are equal to zero. We then have
\[
\dd X_t=QX_t\,\dd t+ \dd M^X_t,
\]
where $M^X$ a martingale with values in $\rr^d$. We also assume that $Q$ is irreducible and we denote by $\pi$ the vector representing the invariant distribution. 

Furthermore it will be throughout assumed that $Y$ and $X$ have no simultaneous jumps, hence the quadratic variation process $[X,Y]$ ($[X,Y]_t=\sum_{s\leq t}\Delta X_s\Delta Y_s$) is identically zero. 
\medskip\\
For the single obligor case, we pose the following model with regime switching,
\[
\dd Y_t= \lambda^\top X_t(1-Y_t)\,\dd t + \dd m_t,
\]
where $\lambda\in\rr^d_+$. 

One way of constructing this model is by realizing it on a product space with $\lambda_t=\lambda^\top X_t$ as in Section~\ref{section:inhomo}.
Alternatively, one can realize $Y$ as standard Poisson process and independently of it, $X$ as a Markov chain on the auxiliary space under $\qq$. By independence, one has $[X,Y]=0$ under $\qq$ and as these brackets remain the same under an absolutely continuous change of measure using the $\qq$-martingale $\mu$ of the previous section, we are then guaranteed to have $[X,Y]=0$ under $\pp$ as well. In this case it is possible to have $\pp$ defined on $(\Omega,\cf)$ for $\cf=\cf_\infty$, where we use the filtration generated by $Y$ and $X$. As a side remark we note that $\pp$ will not be absolutely continuous w.r.t.\ $\qq$ on $\cf_\infty$.
\medskip\\
In all what follows in this paper we adopt the following \emph{Convention: we will use the generic notation $M$ for a martingale, possibly even of varying dimensions, whose precise form is not important.} 
\medskip\\
An important role will be played by the matrices $Q_{k\lambda}:=Q-k\,\diag(\lambda)$ for $k\geq 0$. Here $\diag(\lambda)$ is the diagonal matrix with $ii$-element equal to $\lambda_i$.  
Here is a, possibly known, stability result for the matrix $Q_\lambda$ (we take $k=1$, but a similar result is obviously true for all positive $k$).
\begin{lemma}\label{lemma:qlambda}
Let $\lambda_i>0$ for all $i$. Then the matrix $Q_\lambda$ is invertible and $\exp(Q_\lambda t)\to 0$ for $t\to\infty$.
\end{lemma}

\begin{proof}
That $Q_\lambda$ is invertible, can be seen as follows. Write 
\[
Q_\lambda=-(I-Q\diag(\lambda)^{-1})\diag(\lambda)
\]
and note that $Q\diag(\lambda)^{-1}$ is also the intensity matrix of a Markov chain, as its off-diagonal elements are positive and $\one^\top Q\diag(\lambda)^{-1}=0$. Therefore $I-Q\diag(\lambda)^{-1}$ is invertible, and so is $Q_\lambda$.

In proving the limit result, we give a probabilistic argument.\footnote{This argument has been provided by Koen de Turck, University of Ghent.} Consider the augmented matrix
\[
Q^a_\lambda =
\begin{pmatrix}
0 & -\one^\top Q_\lambda \\
0 & Q_\lambda
\end{pmatrix},
\]
which is the transition matrix of a Markov chain taking values in $\{e^a_0,\ldots,e^a_d\}$, labelled as the standard basis vectors of $\rr^{d+1}$. Clearly, $0$ is an absorbing state, and the only one. Hence whatever initial state $x^a(0)$, we have that $\exp(Q^a_\lambda t)x^a(0)\to e^a_0$ for $t\to\infty$. Computing the exponential and taking $x^a(0)\neq e^a_0$, we find
\[
\exp(Q^a_\lambda t)x^a=
\begin{pmatrix}
1 & \one^\top(I-\exp(Q_\lambda t)) \\
0 & \exp(Q_\lambda t)
\end{pmatrix}x^a(0) = 
\begin{pmatrix}
\one^\top(I-\exp(Q_\lambda t))x(0) \\
\exp(Q_\lambda t)x(0).
\end{pmatrix}
\]
Hence $\exp(Q_\lambda t)\to 0$.
\end{proof}
In a next section, see Remark~\ref{remark:n=1}, we shall see how to compute $\pp(Y_t=1)$. It turns out to be the case that 
\[
\pp(Y_t=1)=1-\one^\top \exp(Q_\lambda t)x(0).
\]
We conclude in view of Lemma~\ref{lemma:qlambda}  that $\pp(Y_t=1)\to 1$ for $t\to\infty$. Hence, with probability one, the obligor eventually defaults, as expected.

\subsection{The MM model for multiple obligors}\label{section:mmmultiple}

In Section~\ref{section:mm1} we have seen results for default processes in the situation of a single obligor. In the present section we generalize those results, at the cost of considerably more complexity, to the situation of multiple obligors.

\subsubsection{Multiple obligors with time-varying intensity}
\label{section:mcl}

Recall \eqref{eq:yi}. Let's first look at the constant intensity case, $\lambda_t=\lambda>0$. Then $N_t=\sum_{i=1}^nY^i_t$ satisfies
\begin{equation}\label{eq:bin}
\dd N_t=\lambda(n-N_t)\,\dd t+ \dd m_t,
\end{equation}
where $m=\sum_{i=1}^nm^i$. By the independence of the default times,
$m$ is a martingale w.r.t.\ $\ff$ and $N_t$ has the $\mathrm{Bin}(n,1-\exp(-\lambda t))$ distribution. Moreover, given $N_u$, $u\leq s$, $N_t-N_s$ has for $t>s$ the $\mathrm{Bin}(n-N_s,1-\exp(-\lambda (t-s)))$ distribution. This model has long ago been used in software reliability going back to \cite{jm}, with various refinements, like in a Bayesian set up the parameters $n$ and $\lambda$ being random, see \cite{ks, littlewood} or with time varying but deterministic intensity function $\lambda(t)$, see \cite{go}.
\medskip\\
Next we look at the case of time varying, possibly random, $\lambda$.
By the assumed conditional independence of the $\tau^i$ given $\lambda$ we have, similar to the constant $\lambda$ case, that  $N_t$, conditional on the process $\lambda$, has a $\mathrm{Bin}(n, 1-\exp(-\Lambda_t))$ distribution with $\Lambda_t=\int_0^t\lambda_s\,\dd s$. 

Let $p^k(t)=\pp(N_t=k|\cf^\lambda)$, put 
\[
p(t)=\begin{pmatrix} p^{0}(t) \\ \vdots \\ p^{n}(t)
\end{pmatrix} 
\]
and 
\begin{equation}\label{eq:A}
A=\begin{pmatrix}
-n   & 0               & \cdots & \cdots & \cdots & 0 \\
n & -(n-1) & 0               & \cdots & \cdots & 0 \\
0             & n-1 & -(n-2) & 0      & \cdots & 0 \\
\vdots        & \ddots          & \ddots          & \ddots & \ddots & \vdots \\
\vdots        &                 & \ddots          & \ddots & -1 & 0 \\
0             & \cdots          & \cdots          & 0      & 1 & 0
\end{pmatrix}.
\end{equation}
Then we have for $p(t)$ the system of differential equations
\[
\dot{p}(t)=\lambda_tAp(t),
\]
which has solution (here we use that $\lambda$ is real-valued)
\[
p(t)=\exp(\Lambda_t A)e_0,
\]
where $\Lambda_t=\int_0^t\lambda_s\,\dd s$ and $e_0$ is the first standard basis vector of $\rr^{n+1}$. For the vector whose elements are the unconditional probabilities $\pp(N_t=k)$ one has to take the expectation and it depends on the specification of $\lambda$ whether this results in analytic expressions. We will see that this happens in case of a Markov modulated rate process.

\subsubsection{The MM case}\label{section:mmmo}

We assume to have a finite state Markov process as in Section~\ref{section:mm1} and let $\lambda_t=\lambda^\top X_{t-}$. For $N_t$ one now has its submartingale decomposition
\[
\dd N_t=\lambda^\top X_t(n-N_t)\,\dd t+ \dd m_t.
\]
This is the model of Section~\ref{section:mm1} extended to more obligors. The default rate for each obligor has become random ($\lambda^\top X_t$), but is taken the same for all of them.

Let $\nu^k_t=\one_{\{N_t=k\}}$, $k=0,\ldots,n$. For notational convenience we set $\nu^{-1}_t=0$. It follows that $\Delta\nu^k_{t}=1$ iff $N_t$ jumps from $k-1$ to $k$ at $t$, and $\Delta\nu^k_{t}=-1$ iff $N_t$ jumps from $k$ to $k+1$. This can be summarized by 
\[
\dd\nu^k_{t}=(\nu^{k-1}_{t-}-\nu^k_{t-})\,\dd N_t.
\]
In vector form this becomes
\begin{equation}\label{eq:nu0}
\dd\nu_t=(J-I)\nu_{t-}\,\dd N_t,
\end{equation}
where
\[
J=\begin{pmatrix}
0      &        &        &        &   \\
1      & 0      &        &        &   \\  
0      & 1      & \ddots &        &   \\
\vdots & \ddots & \ddots & \ddots & 0 \\
0      & \cdots & 0      & 1      & 0
\end{pmatrix}.
\]
Using the dynamics for $N$, we get
\begin{align*}
\dd\nu^k_{t} & =(\nu^{k-1}_{t-}-\nu^k_{t-})(\lambda^\top X_{t-}(n-N_t)\,\dd t+ \dd m_t) \\
& = \lambda^\top X_t ((n-k+1)\nu^{k-1}_t-(n-k)\nu^k_t)\,\dd t +\dd M_t.
\end{align*}
Letting $\nu_t=\begin{pmatrix}
\nu^0_t \\
\vdots \\
\nu^n_t
\end{pmatrix}$, we get from the above display
\begin{equation}\label{eq:nu}
\dd\nu_t=\lambda^\top X_t A\nu_t\,\dd t+ \dd M_t,
\end{equation}
where $A$ is as in \eqref{eq:A}. This equation for $\nu$ is a main ingredient in the next result.

\begin{proposition}\label{prop:zeta}
Let $\zeta_t=\nu_t\otimes X_t$. The process $\zeta$ is Markov with transition matrix $\mathbf{Q}$, where $\mathbf{Q}=(A\otimes\diag(\lambda)+I\otimes Q)$. It follows that $\ee[\zeta_t|\cf_s]=\exp(\bq(t-s))\zeta_s$.
\end{proposition}

\begin{proof}
We will use equation~\eqref{eq:nu} together with the dynamics of $X$.
Using the product rule and the fact that $N$ and $X$ do not jump at the same time and summarizing again all martingale terms again as $M$, we get (recall the multiplication rule $(A\otimes B)(C\otimes D)=(AC)\otimes (BD)$)
\begin{align*}
\dd (\nu_t\otimes X_t) & = \left((A\nu_t\lambda^\top X_t)\otimes X_t+\nu_t \otimes (QX_t)\right)\,\dd t + \dd M_t \\
& = \left((A\nu_t)\otimes (X_t\lambda^\top X_t)+\nu_t \otimes (QX_t)\right)\,\dd t + \dd M_t \\
& = \left((A\nu_t)\otimes (\diag(\lambda) X_t)+I\nu_t \otimes (QX_t)\right)\,\dd t + \dd M_t \\
& = (A\otimes\diag(\lambda)+I\otimes Q)(\nu_t \otimes X_t)\,\dd t + \dd M_t \\
& = \mathbf{Q}(\nu_t \otimes X_t)\,\dd t+ \dd M_t.
\end{align*}
Note that $\zeta_t$ by construction consists of the indicators of the values of the joint process $(\nu,X)$. Hence the equation $\dd\zeta_t=\bq\zeta_t\,\dd t + \dd M_t$ reveals, cf.\ Lemma~1.1 in Appendix~B of~\cite{EAM}, that $\zeta$ (and hence  $(\nu,X)$) is Markov.
\end{proof}
An explicit computation shows 
\begin{equation}\label{eq:bq}
\bq=\begin{pmatrix}
Q_{n\lambda}   & 0               & \cdots & \cdots & \cdots & 0 \\
n\,\diag(\lambda) & Q_{(n-1)\lambda} & 0               & \cdots & \cdots & 0 \\
0             & (n-1)\,\diag(\lambda)  & Q_{(n-2)\lambda} & 0      & \cdots & 0 \\
\vdots        & \ddots          & \ddots          & \ddots & \ddots & \vdots \\
\vdots        &                 & \ddots          & \ddots & Q_\lambda & 0 \\
0             & \cdots          & \cdots          & 0      & \diag(\lambda) & Q
\end{pmatrix},
\end{equation}
where for $k\in\nn$ we have $Q_{k\lambda}=Q-k\,\diag(\lambda)$.

\begin{remark}
The original dynamic equations for $X_t$ and $N_t$ can be retrieved from Proposition~\ref{prop:zeta}.
Realizing the relations $X_t=(\one^\top \otimes I)\zeta_t$ and $(\one^\top \otimes I)\mathbf{Q}=\one^\top\otimes Q$, and $\one^\top A=0$, we obtain from Proposition~\ref{prop:zeta}
\begin{align*}
\dd X_t & =(\one^\top \otimes I)\left(\mathbf{Q}(\nu_t \otimes X_t)\right)\,\dd t+ \dd M_t \\
& = (\one^\top\otimes Q)(\nu_t\otimes X_t)\,\dd t+ \dd M_t \\
& =  QX_t\,\dd t+ \dd M_t.
\end{align*}
Similarly, we get from $\nu_t=(I\otimes \one^\top)\zeta_t$,
\begin{align*}
\dd \nu_t & =(I\otimes \one^\top)\left(\mathbf{Q}(\nu_t \otimes X_t)\right)\,\dd t+ \dd M_t \\
& = (A\otimes \lambda^\top)(\nu_t\otimes X_t)\,\dd t+ \dd M_t \\
& =  A\nu_t\lambda^\top X_t\,\dd t+ \dd M_t.
\end{align*}
Using $\begin{pmatrix} 0 & 1 & \cdots & n\end{pmatrix}A\nu_t=\begin{pmatrix} n &  \cdots & 1 & 0\end{pmatrix}\nu_t=n-N_t$, we get from the last display the decomposition $\dd N_t=(n-N_t)\lambda^\top X_t \,\dd t+ \dd m_t$ back.
\end{remark}
Letting $\pi(t)=\ee \zeta_t$, we obtain from Proposition~\ref{prop:zeta} the ODE 
\begin{equation}\label{eq:pi}
\dot{\pi}(t) =\bq\pi(t)
\end{equation}
with the initial condition $\pi(0)=e_0\otimes x(0)$, where $e_0$ has $1$ as its first element, all other elements being zero. 
We will give a rather explicit expression for $\pi(t)=\exp(\bq t)\pi(0)$, for which we need some additional results.

The differential equation for $\pi$ is the following type of forward equation,
\[
\dot{F}=\mathbf{Q}F.
\]
Here $F$ can be any matrix valued function of appropriate dimensions.
We will block-diagonalize the matrix $\mathbf{Q}$. The transformation that is needed for that is given by the matrix $V$ whose $ij$-block ($i,j=0,\ldots,n$) is
\[
V_{ij}={n-j \choose n-i}(-1)^{i-j}I.
\]
Note that $V_{ij}=0$ for $i<j$, $V$ is block lower-triangular.
The inverse matrix is also block lower-triangular with blocks
\[
V^{-1}_{ij}={n-j \choose n-i}I.
\]
One may check by direct computation that indeed $VV^{-1}=I$.
It is straightforward to verify that $\mathbf{Q}^V:=V^{-1}\mathbf{Q}V$ is block-diagonal with $i$-th block ($i=0,\ldots,n$) equal to
\[
\mathbf{Q}^V_i=Q_{(n-i)\lambda}.
\]
Putting $G=V^{-1}F$ we obtain
\[
\dot{G}=\mathbf{Q}^V\,G,
\]
whose solution satisfying $G(0)=I$ is block diagonal with $i$-th block $G_i(t)=\exp(Q_{(n-i)\lambda}t)$.
We thus obtain the following lemma.
\begin{lemma}\label{lemma:diag}
The solution to the forward ODE $\dot{F}=\mathbf{Q}F$ with initial condition $F(0)$ is given by $F(t)=\exp(\bq t)F(0)$, where
\[
\exp(\bq t)=V
\begin{pmatrix}
\exp(Q_{n\lambda}t) & & \\
& \ddots & \\
& & \exp(Qt)
\end{pmatrix}
V^{-1}.
\]
If $F(t)=\exp(\bq t)$, its blocks $F_{ij}(t)$  can be explicitly computed. One has $F_{ij}(t)=0$ if $i<j$, and for $i\geq j$ it holds that 
\[
F_{ij}(t)={n-j \choose n-i} \sum_{k=j}^i(-1)^{i-k} {i-j\choose i-k}\exp(Q_{(n-k)\lambda}t).
\]
\end{lemma}

\begin{proof}
We use the block triangular structure of $V$ and $V^{-1}$ together with the block diagonal structure of $\bq^V$ to compute
\begin{align*}
F_{ij}(t) & = \sum_{k=j}^{i}V_{ik}\exp(Q_{(n-k)\lambda}t)V_{kj} \\
& = \sum_{k=j}^{i}{n-k \choose n-i}(-1)^{i-k}\exp(Q_{(n-k)\lambda}t){n-j \choose n-k} \\
& = {n-j \choose n-i}\sum_{k=j}^{i}(-1)^{i-k}{i-j \choose i-k}\exp(Q_{(n-k)\lambda}t),
\end{align*}
as stated.
\end{proof}

\begin{proposition}\label{proposition:pink}
The solution $\pi(t)$ to the system \eqref{eq:pi} of ODEs under the initial condition $\pi(0)=e_0\otimes x(0)$ has components $\pi^i(t)\in\rr^d$ given by
\begin{equation}\label{eq:pisol}
\pi^i(t)={n \choose i}\sum_{k=0}^i(-1)^{i-k}{i\choose k}\exp(Q_{(n-k)\lambda}t)x(0).
\end{equation}
\end{proposition}

\begin{proof}
We use Lemma~\ref{lemma:diag} and recall the specific form of the initial condition $\pi(0)$. We have to compute $\exp(\bq t)\pi(0)$ and obtain from Lemma~\ref{lemma:diag} with $j=0$ for $\pi^i(t)=F_{i0}(t)$
\begin{align*}
\pi^i(t) & ={n \choose n-i} \sum_{k=0}^i(-1)^{i-k} {i\choose i-k}\exp(Q_{(n-k)\lambda}t)x(0) \\
& ={n \choose i} \sum_{k=0}^i(-1)^{i-k} {i\choose k}\exp(Q_{(n-k)\lambda}t)x(0). 
\end{align*}
\end{proof}

\begin{remark}\label{remark:n=1}
Let us look at a special case, $n=1$. Then we can write $N_t=Y_t$
and it is sufficient to compute
\begin{equation}\label{eq:zt}
\pi^1(t)=\ee(Y_tX_t)=\left(\exp(Qt)-\exp(Q_\lambda t)\right)x(0).
\end{equation}
As a consequence we are able to compute $\pp(Y_t=1)=\one^\top\ee(Y_tX_t)$,
\[
\pp(Y_t=1)=1-\one^\top \exp(Q_\lambda t)x(0),
\]
since $\one^\top\exp(Q_t)=\one^\top$.
As $\exp(Qt)\to \pi\one^\top$, we conclude in view of Lemma~\ref{lemma:qlambda} from \eqref{eq:zt} that $\pi^1(t)\to\pi$ for $t\to\infty$. This result should be obvious, as $Y_t$ eventually becomes 1 and $X_t$ converges in distribution to its invariant law. 

For the case $n>1$ the expressions for $\pi^i(t)$ are a bit complicated, but their asymptotic values for $t\to\infty$, are as expected, $\pi^i(t)\to 0$ for $i<n$, whereas $\pi^n(t)\to \pi$. This again follows from Lemma~\ref{lemma:qlambda}.
\end{remark}
Proposition~\ref{proposition:pink} has the following corollary.

\begin{corollary}\label{cor:multiphi}
Let $\phi(t,u)=\ee\exp(\ii u N_t)X_t$. It holds that
\[
\phi(t,u)=\sum_{k=0}^n{n\choose k}\exp(\ii uk)(1-\exp(\ii u))^{n-k}\exp(Q_{(n-k)\lambda}t)x(0).
\]
\end{corollary}

\begin{proof}
We shall use the elementary identity
\[
\sum_{k=j}^n\beta^k{n\choose k}{k\choose j}={n\choose j}\beta^j(1+\beta)^{n-j} 
\]
for $\beta=-e^{-\ii u}$ in the last step in the chain of equalities below. From Proposition~\ref{proposition:pink} we obtain 
\begin{align*}
\ee\exp(\ii u N_t)X_t & = \sum_{k=0}^ne^{\ii u k}\pi^k(t) \\
& = \sum_{k=0}^ne^{\ii u k}{n \choose k}\sum_{j=0}^k(-1)^{k-j}{k\choose j}\exp(Q_{(n-j)\lambda}t)x(0) \\
& = \sum_{j=0}^n\sum_{k=j}^n(-e^{\ii u})^k{n \choose k}{k\choose j}(-1)^{j}\exp(Q_{(n-j)\lambda}t)x(0) \\
& = \sum_{j=0}^n {n\choose j}e^{\ii ju}(1-e^{\ii u})^{n-j}\exp(Q_{(n-j)\lambda}t)x(0).
\end{align*}
\end{proof}

\begin{remark}\label{remark:bin}
Alternatively, one can compute a moment generating function $\psi(t,v)=\ee\exp(-vN_t)X_t$ for $v\geq 0$. Let $B$ have a binomial distribution with parameters $n$ and $p=1-\exp(-v)$. Then we have for $\psi(t,v)$ the compact expression $\psi(t,v)=\ee \exp((Q-B\diag(\lambda))t)x(0)=\ee\exp(Q_{\lambda B}t)x(0)$.
\end{remark}

\begin{remark}
There appears to be no simpler representation for $\phi(t,u)$. We note that this function also satisfies the  PDE
\begin{equation}
\dot{\phi}(t,u)=(Q+n(e^{\ii u}-1)\diag(\lambda))\phi(t,u)+\ii(e^{\ii u}-1)\diag(\lambda)\frac{\partial \phi(t,u)}{\partial u}.
\end{equation}
Just by computing the partial derivatives, one verifies that this equation holds. Alternatively, one can apply the It\^o formula to $\exp(\ii u N_t)X_t$ followed by taking expectations. 
\end{remark}

\subsubsection{Conditional probabilities}\label{section:cpf}

The vehicle we use is the process $\zeta$, recall $\zeta_t=\nu_t\otimes X_t$. Our aim is to find expressions for $\zeta_{t|s}=\ee [\zeta_t|\cf_s]$ for $t>s$, from which one can deduce  the conditional probabilities $\ee[\nu_t|\cf_s]$ and $\ee[N_t|\cf_s]$. By the Markov property, Proposition~\ref{prop:zeta}, we have $\ee[\zeta_t|\cf_s]=\exp(\bq(t-s))\zeta_s$. 
Let $\zeta_{t|s}=\ee[\zeta_t|\cf_s]$ and $\zeta^k_{t|s}=\ee[\one_{\{N_t=k\}}X_t|\cf_s]$. 
We aim at a more explicit representation  of the conditional probabilities $\zeta^k_{t|s}$ for $k\geq 0$. Note that $\zeta^k_{t|s}=(e_k^\top\otimes I)\zeta_{t|s}$. Hence $\zeta^k_{t|s}=(e_k^\top\otimes I)\exp(\mathbf{Q}(t-s))\zeta_s$.
Using Lemma~\ref{lemma:diag}, we have
\[
\zeta^k_{t|s}=(e_k^\top\otimes I)V
\begin{pmatrix}
\exp(Q_{n\lambda}(t-s)) & & \\
& \ddots & \\
& & \exp(Q(t-s))
\end{pmatrix}
V^{-1}\zeta_s.
\]
By matrix computations as before this leads to the following result.
\begin{proposition}\label{prop:zetak}
It holds that 
\[
\zeta^k_{t|s}=\sum_{j=0}^k{n-j \choose k-j}\sum_{i=0}^k(-1)^{k-i}{k-j\choose k-i}\exp(Q_{(n-i)\lambda}(t-s))\zeta^j_s.
\]
\end{proposition}
Note that in the formula of this proposition, only one of the $\zeta^j_s$ is different from zero and then equal to $X_s$. Effectively, the sum over $j$ thus reduces to one term only.
The conditional probabilities $\nu^k_{t|s}=\pp(N_t=k|\cf_s)$ can now simply be computed as $\one^\top\zeta^k_{t|s}$. Note that these still depend on $X_s$, and one has the explicit expression 
\[
\ee[\nu^k_t|\cf_s]=
\sum_{j=0}^n{n-j \choose n-k} \sum_{i=j}^k(-1)^{k-i} {k-j\choose k-i}\one^\top\exp(Q_{(n-i)\lambda}(t-s))X_s\nu^j_s.
\]

\begin{remark}
Consider the special case $n=1$ and let $Z_t=Y_tX_t$, $Y_t$ as in Section~\ref{section:mm1}. This amounts to taking $k=n=1$ in Proposition~\ref{prop:zetak} and one gets for $Z_{t|s}=\ee[Z_t|\cf^Y_s]$ the simpler expression
\begin{equation}
Z_{t|s}  = 
\exp(Q_\lambda (t-s))Z_s+\big(\exp(Q(t-s))-\exp(Q_\lambda (t-s))\big)X_s. \label{eq:zts}
\end{equation}
\end{remark}
The next purpose is to compute $\ee[e^{\ii uN_t}X_t|\cf_s]$ and from that one $\ee[e^{\ii uN_t}|\cf_s]=\one^\top\ee[e^{\ii uN_t}X_t|\cf_s]$. 
\begin{proposition}
The following hold.
\begin{align}
\ee[e^{\ii uN_t}X_t|\cf_s] &=\sum_{k=0}^n\sum_{j=k}^n {n-k \choose j-k}(1-e^{\ii u})^{n-j}e^{\ii uj}\exp(Q_{(n-j)\lambda} (t-s))\zeta^k_s, \nonumber\\
\ee[e^{\ii uN_t}|\cf_s]&=\sum_{k=0}^n\sum_{j=k}^n {n-k \choose j-k}(1-e^{\ii u})^{n-j}e^{\ii uj}\one^\top\exp(Q_{(n-j)\lambda} (t-s))\zeta^k_s.\label{eq:cfmmn}
\end{align}
\end{proposition}

\begin{proof}
We start from the identity $e^{iuN_t}X_t=\mathbf{F}\zeta_t$, with 
$
\mathbf{F}=e(u)\otimes I,
$
where $e(u)=\begin{pmatrix} 1 & e^{\ii u} & \cdots & e^{n\ii u} \end{pmatrix}$. Hence we have
\[
\ee[e^{\ii uN_t}X_t|\cf_s]=(e(u)\otimes I)\exp(\mathbf{Q}(t-s))\zeta_s.
\]
This can be put into the asserted more explicit representation, involving the matrices $Q_{k\lambda}$ by application of Proposition~\ref{prop:zetak}. The second assertion is a trivial consequence. 
\end{proof}
It is conceivable that only $N$ is observed, and not the background process $X$. In such a case one is only able to compute conditional expectation of quantities as above  conditioned on $\cf^N_s$ instead of $\cf_s$. See Section~\ref{section:filter2} for results.

\section{The Markov Modulated Poisson process}\label{section:mmpoisson}

In this section we study MM Poisson processes. These have an intensity process $\lambda_t=\lambda^\top X_t$, using the same notation as before. In terms of defaultable obligors, such processes occur as limits of the total number of defaults $N_t$ as in Section~\ref{section:mmmultiple} where $n\to\infty$ and the vector $\lambda$ is scaled to become $\lambda/n$, as we shall see later. So we can use this to approximate the total number of defaults in a market with a large number of obligors, where each of them has small default rate.

\subsection{The model}
The point of departure is to postulate the dynamics of the counting process $N$ as
\[
\dd N_t=\lambda^\top X_t\,\dd t+ \dd m_t.
\]
We follow the same approach as before. So
we use that
conditionally on $\cf^X$ we have that $N_t$ has a $\mathrm{Poisson}(\Lambda_t)$ distribution with $\Lambda_t=\int_0^t\lambda^\top X_s\,\dd s$. It follows that 
\[
\ee[\one_{\{N_t=k\}}X_t|\cf^X]=\frac{1}{k!}\Lambda_t^k\exp(- \Lambda_t)X_t=:p^{k}(t)X_t,
\]
and 
\[
\frac{\dd}{\dd t} p^{k}(t)=p^{k-1}(t)-p^{k}(t)\lambda^\top X_t.
\]
Then we obtain
\[
\dd \ee[\one_{\{N_t=k\}}X_t|\cf^X]=\big(p^{k-1}(t)-p^{k}(t)\big)\diag(\lambda)X_t\,\dd t +p^{k}(t)(QX_t\,\dd t + \dd M_t),
\]
and with $\pi^{k}(t)=\ee(p^{k}(t)X_t)$ we find
\begin{align*}
\dot{\pi}^{k}(t) & =\diag(\lambda)\pi^{k-1}(t)+(Q-\diag(\lambda))\pi^{k}(t).
\end{align*} 
For $k=0$, one immediately finds the solution $\pi^0(t)=\exp(Q_\lambda t)x(0)$. For $k>0$ there seems to be no simply expression in terms of exponential of $Q$ and $Q_{k\lambda}$ as in Proposition~\ref{proposition:pink}, not even for $k=1$, although one has
\[
\pi^1(t)=\int_0^t\exp(-Q_\lambda(t-s))\diag(\lambda)\exp(Q_\lambda s)\,\dd s\,x(0).
\]
However, it is possible to get a formula for the vector
\[
\Pi^n(t) = \begin{pmatrix}
\pi^0(t) \\
\vdots \\
\pi^n(t)
\end{pmatrix},
\]
since it satisfies the ODE
\[
\dot{\Pi}^n(t)=\bq_n\Pi^n(t),
\]
where $\bq_n\in\rr^{(n+1)d\times (n+1)d}$ is given by
\[
\bq_n=
\begin{pmatrix}
Q-\diag(\lambda) & 0  & \cdots & \cdots & 0 \\
\diag(\lambda) & Q-\diag(\lambda) & 0 & & 0 \\
0 & \diag(\lambda) & \ddots & \ddots & \vdots \\
\vdots & & \ddots & Q-\diag(\lambda) & 0 \\
0 & \cdots & 0 & \diag(\lambda) & Q-\diag(\lambda)
\end{pmatrix}.
\]
Together with the initial conditions $\pi^k(0)=\delta_{k0}x(0)$, one obtains
\[
\Pi^n(t)=\exp(\bq_nt)(e^n_0\otimes x(0)),
\]
where $e^n_0$ is the first basis vector of $\rr^{n+1}$.
An elementary expression for $\exp(\bq_nt)$ is not available due to the fact that $Q-\diag(\lambda)$ and $\diag(\lambda)$ do not commute. Besides, $\bq_n$ is block lower triangular with identical blocks on the main diagonal and therefore cannot be block diagonalized.
\medskip\\
However, in the present case there is a nice expression for the characteristic function $\phi(t,u)=\ee\exp(\ii uN_t)X_t$, unlike the situation of Corollary~\ref{cor:multiphi}. To determine $\phi(t,u)$, we apply the It\^o formula (note that $[N,X]=0$) and obtain
\begin{equation}\label{eq:phinx}
\dd \exp(\ii uN_t)X_t = (e^{\ii u}-1)e^{\ii u N_{t-}}X_{t-}\dd N_t + e^{\ii u N_{t-}}\dd X_{t},
\end{equation}
which yields after taking expectations and using the dynamics of $X$ and $N$
\[
\dot{\phi}(t,u)=((e^{\ii u}-1)\diag(\lambda)+Q)\phi(t,u).
\]
Hence
\[
\phi(t,u)=\exp\big(((e^{\ii u}-1)\diag(\lambda)+Q)t\big)x(0).
\]
Contrary to the $\pi^k(t)$ of Proposition~\ref{proposition:pink} we thus found a \emph{simple} formula for $\phi(t,u)$. This formula is in line with \cite[Proposition~1.6]{ASM} for Markovian arrival processes.

\begin{remark}
It is possible to obtain the above results as limits from results in Section~\ref{section:mmmo}, by replacing there $\lambda$ by $\lambda/n$ and letting $n\to\infty$. 

If we look at the moment generating functions $\psi(t,v)=\ee\exp(-vN_t)X_t$, we have $\psi(t,v)=\exp\big((Q-(1-e^{-v})\diag(\lambda))t\big)x(0)$. Replace in Remark~\ref{remark:bin} the parameter $\lambda$ with $\lambda/n$ and let $n\to\infty$ and write $B_n$ instead of $B$. Then we have $\psi_n(t,v)=\ee\exp\big((Q-\diag(\lambda)B_n/n)t\big)x(0)$. As $B_n/n\to 1-e^{-v}$ a.s., we obtain $\exp\big((Q-\diag(\lambda)B_n/n)t\big)\to\exp\big((Q-\diag(\lambda)(1-e^{-v}))t\big)$ a.s. Since the exponentials are bounded, we also have convergence of the expectations by dominated convergence. Replacing $-v$ with $\ii u$ gives the characteristic function.
\end{remark}

\subsection{Conditional probabilities}

Mimicking the approach of Section~\ref{section:mmmo}, we consider again the $\nu^k_t=\one_{\{N_t=k\}}$. Let   
\[
\bnu^n_t=\begin{pmatrix}
\nu^0_t \\
\vdots \\
\nu^n_t
\end{pmatrix}.
\]
Then $\bnu^n$ still satisfies Equation~\eqref{eq:nu0}. Combining this with the dynamics of $N$, we obtain the semimartingale decomposition
\[
\dd\bnu^n_t=\lambda^\top X_t(J-I)\bnu^n_t\,\dd t + \dd M_t.
\]
Letting $\bzeta^n_t=\bnu^n_t\otimes X_t$, then we can derive, similar to the approach of Section~\ref{section:mmmo},
\[
\dd\bzeta^n_t= \bq_n\bzeta^n_t\,\dd t + \dd M_t.
\]
This is for each $n$ a finite dimensional system, which can be extended to an infinite dimensional system for $\zeta_t$. The resulting infinite coefficient matrix will be lower triangular again,
\[
\dd\zeta_t=\bq_\infty\zeta_t\,\dd t +\dd M_t,
\]
where $\bq_\infty=I_\infty\otimes Q_\lambda-J_\infty\otimes \diag(\lambda)$ with $I_\infty$ the infinite dimensional identity matrix and $J_\infty$ the infinite dimensional counterpart of the earlier encountered matrix $J$.
It follows that for the vector of conditional probabilities we have 
\[
\ee[\zeta_t|\cf_s]=\exp(\bq_\infty(t-s))\bzeta_s. 
\]
This looks like an infinite dimensional expression, but  $\ee[\one_{\{N_t=n\}}X_t|\cf_s]$ can be computed from $\ee[\bar\zeta^n_t|\cf_s]=\exp(\bq_n(t-s))\bar\zeta^n_s$, which effectively reduces the infinite dimensional system to a finite dimensional one.
One can now also compute, with $\ell_n^\top=\begin{pmatrix} 0 & \cdots & 0 & 1 \end{pmatrix}\in\rr^{1\times (n+1)}$, 
\[
\pp(N_t=n,X_t=e_j|\cf_s)=(\ell_n^\top\otimes e_j^\top)\exp(\bq_n(t-s))\bzeta^n_s.
\]

\subsection{Conditional characteristic function}

Our aim is to find an expression for $\phi_{t|s}:=\ee[\exp(\ii uN_t)X_t|\cf_s]$. Since we deal in the present section with the MM Poisson process $N$, the bivariate process $(X,N)$, unlike its counterpart in Section~\ref{section:mm},  is an instance of a Markov additive process~\cite{ASM}, and $\ee[\exp(\ii u(N_t-N_s))X_t|\cf_s]$ will only depend on $X_s$. 
We first follow the forward approach.  
\begin{proposition}
It holds that
\begin{equation}\label{eq:phits}
\phi_{t|s}=\exp\left(((e^{\ii u}-1)\diag(\lambda) + Q)(t-s)\right)e^{\ii u N_{s}}X_s.
\end{equation}
\end{proposition}

\begin{proof}
Starting point is Equation~\eqref{eq:phinx}. We use the dynamics of $N$ and $X$ to get the semimartingale decomposition
\begin{align*}
\dd \exp(\ii uN_t)X_t & = (e^{\ii u}-1)e^{\ii u N_{t}}\diag(\lambda)X_t\,\dd t + e^{\ii u N_{t}}QX_t\,\dd t +\dd M_{t}\\
& = ((e^{\ii u}-1)\diag(\lambda) + Q)e^{\ii u N_{t}}X_t\,\dd t +\dd M_{t}.
\end{align*}
Let $t\geq s$. We obtain (differentials w.r.t.\ $t$)
\[
\dd\phi_{t|s}=((e^{\ii u}-1)\diag(\lambda) + Q)\phi_{t|s}\,\dd t,
\]
which has the desired solution.
\end{proof}
Next we outline the backward approach. Observe first that $\phi_{t|s}$ is a martingale in the $s$-parameter and that due to the fact that $(N,X)$ is Markov, we can write for some function $\Phi$, $\phi_{t|s}=\Phi(t-s,N_s)X_s$. We identify $\Phi$ as follows, using the It\^o formula w.r.t.\ $s$. We obtain
\begin{align*}
\dd \phi_{t|s} & = \left(-\dot{\Phi}(t-s,N_s)\,\dd s+ (\Phi(t-s,N_{s-}+1)-\Phi(t-s,N_{s-}))\dd N_s\right)X_{s-} \\
& \quad\mbox{}\quad +\Phi(t-s,N_{s-})\,\dd X_s \\
& = \left(-\dot{\Phi}(t-s,N_s)+ (\Phi(t-s,N_{s}+1)-\Phi(t-s,N_{s}))\diag(\lambda)\right)X_{s}\,\dd s \\
& \quad\mbox{}\quad +\Phi(t-s,N_{s})Q X_s\,\dd s+\dd M_s.
\end{align*}
The above mentioned martingale property leads to the system of ODEs ($n\geq 0$)
\begin{equation}\label{eq:phin}
\dot{\Phi}(t,n)=\Phi(t,n+1)\diag(\lambda)+\Phi(t,n)\left(Q-\diag(\lambda)\right).
\end{equation}
We have the initial conditions $\Phi(0,n)=\exp(\ii u n)$. To know $\Phi(t,n)$ it seems necessary to know $\Phi(t,n+1)$, which suggest that the ODEs are difficult to solve constructively. Instead, we pose a solution, 
we will verify that 
\[
\Phi(t,n)=\exp\left(((e^{\ii u}-1)\diag(\lambda) + Q)t\right)e^{\ii u n}.
\]
Differentiation of the given expression for $\Phi(t,n)$ gives
\[
\dot{\Phi}(t,n)= \Phi(t,n)((e^{\ii u}-1)\diag(\lambda) + Q).
\]
Note that $\Phi(t,n+1)=\Phi(t,n)e^{\ii u}$. Insertion of this into the ODE gives
\[
\dot{\Phi}(t,n)=\Phi(t,n)(e^{\ii u}\diag(\lambda)+\left(Q-\diag(\lambda)\right)),
\]
which coincides with \eqref{eq:phin}.

\section{Filtering}\label{section:filtering}

Let $N$ be a counting process with predictable intensity process $\lambda$. In many cases it is conceivable that $\lambda$ is an unobserved process and expressions in terms of $\lambda$ are not always useful. Let $\hat\lambda_t=\ee[\lambda_t|\cf^N_t]$. Then the semimartingale decomposition of $N$ w.r.t.\ the filtration $\mathbb{F}^N$ is given by
\[
\dd N_t=\hat\lambda_t\,\dd t + \dd\hat m_t,
\]
where $\hat m$ is a (local) martingale w.r.t.\ $\mathbb{F}^N$.
The general filter of the Markov chain $X$, $\hat X_t=\ee[X_t|\cf^N_t]$ satisfies the following well known formula (see~\cite{bremaud}, originating from~\cite{schuppen}) with $Q$ as in Section~\ref{section:mm1}
\[
\dd \hat{X}_t=Q\hat{X}_t\,\dd t +\hat{\lambda}_{t-}^+(\widehat{X\lambda}_{t-}-\hat{X}_{t-}\hat{\lambda}_{t-})(\dd N_t-\hat\lambda_t\,\dd t),
\] 
where $\widehat{X\lambda}_{t} =\ee[X_t\lambda_t|\cf^N_t]$ and where we use the notation $x^+=\one_{x\neq 0}/x$ for a real number $x$.
For any of the previously met models for the counting process $N$ we have a predictable intensity process of the form $\lambda_t=\lambda^\top X_{t-} f(N_{t-})$, where $f$ depends on the specific model at hand. It follows that $\hat\lambda_t=\lambda^\top\hat{X}_{t-}f(N_{t-})$. In all cases we consider it happens that $f(N_t)$ remains zero after it has reached zero, and hence $N$ stops jumping as soon as $f(N_t)=0$.  Since $\lambda^\top X_t>0$, with the convention $\frac{0}{0}=0$ the above filter equation reduces to 
\begin{equation}\label{eq:filter}
\dd \hat{X}_t=Q\hat{X}_t\,\dd t +\frac{1}{\lambda^\top\hat{X}_{t-}}(\diag(\lambda)\hat{X}_{t-}-\hat{X}_{t-}\lambda^\top\hat{X}_{t-})(\dd N_t-\hat\lambda_t\,\dd t).
\end{equation}
For the specific models we have encountered we give in the next sections more results on $\hat{X}$.

\subsection{Filtering for the MM multiple point process}\label{section:filter2}

The notation of this section is as in Section~\ref{section:mmmo} and subsequent sections. Let $\hat{\zeta}_t=\ee[\zeta_t|\cf^N_t]$. Then $\hat{\zeta}_t=\nu_t\otimes \hat{X}_t$, where $\hat{X}_t=\ee[X_t|\cf^N_t]$. For $\hat{X}_t$ we have from \eqref{eq:filter},
\[
\dd\hat{X}_t=Q\hat{X}_t\,\dd t+\frac{1}{\lambda^\top \hat{X}_{t-}}\left(\diag(\lambda)\hat{X}_{t-}- \hat{X}_{t-}\hat{X}_{t-}^\top\lambda\right)\,(\dd N_t-(n-N_t)\lambda^\top \hat{X}_{t}\,\dd t).
\]
At the jump times $\tau_k$ ($k=1,\ldots,n$) (these are the order statistics of the original default times $\tau^i$) of $N$ we thus have 
\[
X_{\tau_k}=\frac{1}{\lambda^\top \hat{X}_{\tau_k-}}\diag(\lambda)\hat{X}_{\tau_k-}
\]
Between the jump times, $\hat{X}$ evolves according to the ODE
\[
\frac{\dd\hat{X}_t}{\dd t}=Q\hat{X}_t-(n-N_t)(\diag(\lambda)\hat{X}_{t-}- \hat{X}_{t-}\hat{X}_{t-}^\top\lambda),
\]
which is also valid after the last jump of $N$. It follows that for $t\geq \tau_n$ we have $\hat{X}_t=\exp(Q(t-\tau_n))\hat{X}_{\tau_n}$.
\medskip\\
Below we need $[\nu,\hat{X}]^\otimes_t=\sum_{s\leq t}\Delta \nu_s\otimes \Delta\hat{X}_s$. Using the equations for $\nu$ and $\hat{X}$, we find
\[
\dd[\nu,\hat{X}]^\otimes_t=\frac{1}{\lambda^\top\hat{X}_{t-}}((J-I)\otimes (\diag(\lambda)-\lambda^\top\hat{X}_{t-}I))\hat{\zeta}_{t-}\dd N_t.
\]
For $\hat{\zeta}_t$ we have, using the product formula for tensors,
\[
\dd\hat{\zeta}_t = \dd\nu_t\otimes \hat{X}_{t-}+\nu_{t-}\otimes \dd\hat{X}_t+\dd[\nu,\hat{X}]^\otimes_t.
\]
This yields after some tedious computations the following semimartingale decomposition for $\hat\zeta$
\begin{align*}
\dd\hat{\zeta}_t & =\big(I\otimes Q+(n-N_t)(J-I)\otimes\diag(\lambda)\big)\hat{\zeta}_{t}\,\dd t \\
& \quad \mbox{ } +\frac{1}{\lambda^\top\hat{X}_{t-}}\big(J\otimes\diag(\lambda)-\lambda^\top\hat{X}_{t-}I\otimes I\big)\hat{\zeta}_{t-}\,\dd \hat{m}_t \\
& = \bq\hat\zeta_t\,\dd t + \frac{1}{\lambda^\top\hat{X}_{t-}}\big(J\otimes\diag(\lambda)-\lambda^\top\hat{X}_{t-}I\otimes I\big)\hat{\zeta}_{t-}\,\dd \hat{m}_t,
\end{align*}
where $\dd \hat{m}_t=\dd N_t-(n-N_t)\lambda^\top \hat{X}_{t}\,\dd t$ and $\bq$ as in Section~\ref{section:mmmo}.
\medskip\\
Here are two applications. One can now compute 
\[
\pp(N_t=k|\cf^N_s)=\one^\top\ee[\zeta^k_{t|s}|\cf^N_s]=\one^\top\hat\zeta^k_{t|s}, 
\]
for which we can  use $\hat\zeta_{t|s}=\exp(\bq(t-s))\hat\zeta_s$.
Formula \eqref{eq:cfmmn}  yields for the conditional characteristic function of $N_t$ given its own past until time $s<t$ the explicit expression
\[
\ee[e^{\ii uN_t}|\cf^N_s]=\sum_{k=0}^n\sum_{j=k}^n {n-k \choose j-k}(1-e^{\ii u})^{n-j}e^{\ii uj}\one^\top\exp(Q_{(n-j)\lambda} (t-s))\hat{X}_s\nu^k_s.
\]
In case $n=1$ the above formulas simplify considerably. Here are a few examples, where we use the notation of Section~\ref{section:mm1}. Suppose that only $Y$ is observed. Let $\cf^Y_t=\sigma(Y_s,0\leq s\leq t)$. With $Z_t:=Y_tX_t$ we want to compute  $\hat{Z}_{t|s} :=\ee[Z_t|\cf^Y_s]$ for $t\geq s$. Let $\hat{X}_t=\ee[X_t|\cf^Y_t]$, then obviously, $\hat{Z}_{t|s}=\hat{X}_{t|s}Y_s$. Moreover,  one has from \eqref{eq:zts}
\begin{align*}
\hat{Z}_{t|s} 
& = \exp(Q(t-s))\hat{X}_s-\exp(Q_\lambda (t-s))\hat{X}_{s}(1-Y_s).
\end{align*}
As a consequence we have for $\hat{Y}_{t|s}=\one^\top \hat{Z}_{t|s}$
\[
\hat{Y}_{t|s}= 1-\one^\top \exp(Q_\lambda (t-s))\hat{X}_{s}(1-Y_s).
\]

\subsection{Filtering for the MM Poisson process}

The filter equations now take the familiar form
\[
\dd\hat{X}_t=Q\hat{X}_t\,\dd t+\frac{1}{\lambda^\top \hat{X}_{t-}}\left(\diag(\lambda)\hat{X}_{t-}- \hat{X}_{t-}\hat{X}_{t-}^\top\lambda\right)\,(\dd N_t-\lambda^\top \hat{X}_{t}\,\dd t).
\]
For $\bar\nu_t$ we have the infinite dimensional analogue of \eqref{eq:nu0}. This leads for $\hat\zeta_t=\bar\nu_t\otimes \hat{X}_t$ as in a Section~\ref{section:filter2} to
\[
\dd\hat{\zeta}_t  = \bq_\infty\hat\zeta_t\,\dd t + \frac{1}{\lambda^\top\hat{X}_{t-}}\big(J_\infty\otimes\diag(\lambda)-\lambda^\top\hat{X}_{t-}I_\infty\otimes I_\infty\big)\hat{\zeta}_{t-}\,(\dd N_t-\lambda^\top \hat{X}_{t}\,\dd t).
\]
Note that this system is infinite dimensional, but for each $n$ we also have for $\hat{\bar{\zeta}}^n_t=\ee[\hat\zeta^n_t|\cf^N_t]$ the truncated finite dimensional system
\[
\dd\hat{\bar{\zeta}}^n_t  = \bq_n\hat{\bar\zeta}^n_t\,\dd t + \frac{1}{\lambda^\top\hat{X}_{t-}}\big(J\otimes\diag(\lambda)-\lambda^\top\hat{X}_{t-}I\otimes I\big)\hat{\bar{\zeta}}^n_{t-}\,(\dd N_t-\lambda^\top \hat{X}_{t}\,\dd t).
\]
For the conditional characteristic function $\ee[\exp(\ii u N_t)X_t|\cf^N_s]$ we have
\[
\ee[\exp(\ii u N_t)X_t|\cf^N_s]=\exp\left(((e^{\ii u}-1)\diag(\lambda) + Q)(t-s)\right)e^{\ii u N_{s}}\hat{X}_s,
\]
whereas $\psi_t=e^{\ii u N_{t}}\hat{X}_t$ satisfies the equation ($\dd\hat m_t=\dd N_t-\lambda^\top\hat X_t\,\dd t$)
\[
\dd\psi_t=(\frac{e^{\ii u}}{\lambda^\top\hat X_{t-}}\diag(\lambda)-I)\psi_{t-}\dd\hat m_t + \left(Q+(e^{\ii u}-1)\diag(\lambda)\right)\psi_t\,\dd t.
\]

\section{Rapid switching}\label{section:rapid}

In this section we present some auxiliary results that we shall use in obtaining limits for the various default processes when the Markov chain evolves under a rapid switching regime, i.e.\ the transition matrix $Q$ will be replaced with $\alpha Q$, where $\alpha >0$ tends to infinity. In the first two results and their proofs we use the notation $C(M)$ for the matrix of cofactors of a square matrix $M$. Throughout this section we write $\lambda_\infty$ for $\lambda^\top\pi$.

\begin{lemma}\label{lemma:C}
Let $Q$ have a unique  invariant vector $\pi$. Then 
\[
C(Q)=q\,\pi\one^\top,
\]
where the constant $q$ can be computed as $\det(\hat{Q})$, where $\hat{Q}$ is obtained from $Q$ by replacing its last row with $\one^\top$.
\end{lemma}

\begin{proof}
Note first that $\pi$ can be obtained as the solution to $\hat{Q}\pi=e_d$, where $e_d$ is the last basis vector of $\rr^d$. By Cramer's rule $\pi$ can be expressed using the cofactors of $\hat{Q}$. In particular, $\pi_d=\hat{C}_{dd}/\det(\hat{Q})$, where $\hat{C}$ is the cofactor matrix of $\hat{Q}$. But $\hat{C}_{dd}=C_{dd}$, so $\pi_d=C_{dd}/\det(\hat{Q})$.

Write $C=C(Q)$ and recall that $CQ=\det(Q)$ and hence zero. It follows that every row of $C$ is a left eigenvector of $Q$. Since $Q$ has rank $d-1$ by its assumed irreducibility, every row of $C$ is a multiple of $\one^\top$. Hence $C=\alpha\one^\top$, for some $\alpha\in\rr^{d\times 1}$. By similar reasoning,  $C=\pi\beta$ for some $\beta\in\rr^{1\times d}$. We conclude that $C=q\pi\one^\top$ for some real constant $q$. Use now $C_{dd}=q\pi_d$ and the above expression for $\pi_d$ to arrive at $q=\det(\hat{Q})$.
\end{proof}

\begin{proposition}\label{prop:inv}
Let $Q$ have a unique invariant vector $\pi$ and let all $\lambda_i$ be positive. Then $(\alpha Q-\diag(\lambda))^{-1}\to -\frac{\pi\one^\top}{\lambda_\infty}$ for $\alpha\to\infty$.
\end{proposition}

\begin{proof}
We have seen in Section~\ref{section:mm1} that $Q-\diag(\lambda)$ is invertible if all $\lambda_i>0$ and so the same is true for $\alpha Q-\diag(\lambda)$. Both  $\det(\alpha Q-\diag(\lambda))$ and  the cofactor matrix of $\alpha Q-\diag(\lambda)$ are polynomials in $\alpha$ and we compute the leading term. The determinant is computed by summing products of elements of $\alpha Q-\diag(\lambda)$, from each row and each column one. The $\alpha^d$ term in this determinant has coefficient $\det(Q)$, which is zero. Consider the term with $\alpha^{d-1}$. It is seen to be equal to $-\sum_{i=1}^d\lambda_iC(\alpha Q-\diag(\lambda))_{ii}=-\alpha^{d-1}\sum_{i=1}^d\lambda_iC(Q-\diag(\lambda/\alpha))_{ii}$. For the cofactor matrix itself a similar procedure applies. We get $C(\alpha Q-\diag(\lambda))=\alpha^{d-1}C(Q-\diag(\lambda)/\alpha)$ and it results from Lemma~\ref{lemma:C} that for $\alpha\to\infty$
\[
\frac{C(\alpha Q-\diag(\lambda))}{\det(\alpha Q-\diag(\lambda))}\to\frac{C(Q)}{-\sum_{i=1}^d\lambda_iC(Q)_{ii}}=-\frac{q\pi\one^\top}{q\sum_{i=1}^n\lambda_i\pi_i}=-\frac{\pi\one^\top}{\lambda_\infty}.
\]
\end{proof}

\begin{proposition}\label{proposition:exp}
For $\alpha\to\infty$ it holds that 
\[
\exp\big((\alpha Q-\diag(\lambda))t\big)\to\exp(-\lambda_\infty t)\pi\one^\top.
\]
\end{proposition}

\begin{proof} 
For any analytic function $f:\cc\to\cc$, $f(z)=\sum_{k=0}^\infty a_kz^k$, one defines $f(M):=\sum_{k=0}^\infty a_kM^k$ for $M\in\cc^{d\times d}$ (assuming that the power series converges on the spectrum of $M$). It then holds (see also Higham~\cite[Definition~1.11]{higham}, where this is taken as a definition of $f(M)$) that
\[
f(M)=\frac{1}{2\pi\ii}\oint_\Gamma (zI-M)^{-1}f(z)\,\dd z,
\] 
where $\Gamma$ is a closed contour such that all eigenvalues of $M$ are inside it. Take $M=\alpha Q-\diag(\lambda)$. It follows from Proposition~\ref{prop:inv}, note that also $\lambda_\infty$ lies inside $\Gamma$ as it is a convex combination of the $\lambda_i$, that $(zI-\alpha  Q+\diag(\lambda))^{-1}\to\frac{1}{z+\lambda_\infty}\pi\one^\top$. Hence
\[
f(\alpha Q-\diag(\lambda))\to\pi\one^\top f(-\lambda_\infty).
\]
Apply this to $f(z)=\exp(tz)$.
\end{proof}

\subsection{Rapid switching for the MM multiple point process}

Suppose we scale the $Q$ matrix with $\alpha\geq 0$, and we let $X^\alpha$ be Markov with transition matrix $\alpha Q$. Many (random) variables below will be indexed by $\alpha$ as well. Here is a way to get accelerated dynamics for $N^\alpha_t$ (previously denoted $N_t$).

Suppose that one takes the original Markov chain $X$ and replaces the dynamics of $N$ with one in which $X$ is accelerated,
\begin{equation}\label{eq:yalpha}
N^\alpha_t=\int_0^t(n-N^\alpha_s)\lambda^\top X_{\alpha s}\,\dd s +m_t.
\end{equation}
Indeed the process $X^\alpha$ defined by $X^\alpha_t=X_{\alpha t}$ has intensity matrix $\alpha Q$, and its invariant measure is $\pi$ again. Recall that, conditionally on $\cf^X$, $N^\alpha_t$ has a Bin$(n,1-\exp(-\int_0^t\lambda^\top X_{\alpha s}\,\dd s))$ distribution and that its unconditional distribution is Bin$(n,1-\ee\exp(-\int_0^t\lambda^\top X_{\alpha s}\,\dd s))$. 

The ergodic property of $X$ gives $\int_0^tX_{\alpha s}\,\dd s = \frac{1}{\alpha}\int_0^{\alpha t}X_s\,\dd s\to \pi t$ a.s.\ and hence by dominated convergence for the expectations  we have that the limit distribution of $N^\alpha_t$ for $\alpha\to\infty$ is Bin$(n,1-\exp(-\lambda_\infty t))$. One immediately sees that the default times $\tau^{\alpha,k}$ convergence in distribution to $\tau^k$ that are independent and have an exponential distribution with parameter $\lambda_\infty$. 
Keeping this in mind, the other results in this section are easily understandable. 
\medskip\\
We  recall the content of Proposition~\ref{proposition:exp}. Replacing $\lambda$ with $k\lambda$ for $k\geq 0$ (zero included) yields
\begin{equation}\label{eq:qn}
\exp\big((\alpha Q-k\diag(\lambda))t\big)\to\exp(-k\lambda_\infty t)\pi\one^\top.
\end{equation}
To express the dependence of the matrix $\bq$ given by \eqref{eq:bq} on $\alpha $ in the present section, we write $\bq^\alpha $ (so $\bq^\alpha=A\otimes \diag(\lambda) + I\otimes \alpha Q$) and $F^\alpha (t)$ instead of $F(t)$ as given in Lemma~\ref{lemma:diag}. 

\begin{lemma}\label{lemma:f}
The solution $F^\alpha $ to the equation $\dot{F}=\bq^\alpha F$, has for $\alpha \to\infty$  limit $F^\infty$ given by its blocks
\[
F^\infty_{ij}(t)=f^\infty_{ij}(t)\pi\one^\top,
\]
where the $f^\infty_{ij}(t)$ are the binomial probabilities on $n-i$ `successes' of a Bin$(n-j,\exp(-\lambda_\infty t))$ distribution, 
\[
f^\infty_{ij}(t)={n-j \choose n-i} \exp(-(n-i)\lambda_\infty t)(1-\exp(-\lambda_\infty t))^{i-j}.
\]
\end{lemma}

\begin{proof}
We depart from Lemma~\ref{lemma:diag} and the expression for $F^\alpha_{ij}(t)$ given there when we replace $Q$ with $\alpha Q$. Taking limits for $\alpha \to\infty$ yields
\begin{align*}
F^\infty_{ij}(t) & = {n-j \choose n-i} \sum_{k=j}^i(-1)^{i-k} {i-j\choose i-k}\exp(-(n-k)\lambda_\infty t)\pi\one^\top \\
& = {n-j \choose n-i} (-1)^{i-j}\exp(-(n-j)\lambda_\infty t)\sum_{l=0}^{i-j} {i-j\choose l}(-\exp(\lambda_\infty t))^l\pi\one^\top \\
& = {n-j \choose n-i} \exp(-(n-i)\lambda_\infty t)(1-\exp(-\lambda_\infty t))^{i-j}\pi\one^\top,
\end{align*}
from which the assertion follows.
\end{proof}

\begin{remark}
One can also use this proposition to show that $N^\alpha_t$  in the limit has the  Bin$(n,1-\exp(-\lambda_\infty t))$ distribution. Indeed,
since $\nu^i_0=\delta_{i0}$, we get $\pp(N^\alpha_t=i,X_t=e_j)\to F^\infty_{i0}(t)=f^\infty_{i0}(t)\pi$ and hence $\pp(N^\alpha_t=i)\to f^\infty_{i0}(t)$. 
\end{remark}
For conditional probabilities one has the following result.

\begin{corollary}
Let $N$ be a process like in Equation~\eqref{eq:bin}, with $\lambda$ replaced with $\lambda_\infty$. For $\alpha \to\infty$ one has in the limit $\zeta^i_{t|s}=0$ for $i<N_s$ and  for $i\geq N_s$
\[
\zeta^i_{t|s}={n-N_s \choose n-i} \exp(-(n-i)\lambda_\infty\,(t-s))(1-\exp(-\lambda_\infty\,(t-s)))^{i-N_s}   \pi.
\]
It follows that, conditional on $\cf_s$, $N_t-N_s$ has a Bin$(n-N_s,1-\exp(-\lambda_\infty\,(t-s)))$ distribution. In fact, one has weak convergence of the $N^\alpha$ to $N$.
\end{corollary}

\begin{proof}
We compute in the limit $\zeta^i_{t|s}=\ee[\nu^i_tX_t|\cf_s]$ and obtain from Lemma~\ref{lemma:f} 
\begin{align*}
\zeta^i_{t|s} &  = \sum_{j=0}^nF^\infty_{ij}(t-s)\zeta^j_s\\
& = \sum_{j=0}^nf^\infty_{ij}(t-s)\nu^j_s\pi\\
& = \sum_{j=0}^i {n-j \choose n-i} \exp(-(n-i)\lambda_\infty\,(t-s))(1-\exp(-\lambda_\infty\,(t-s)))^{i-j}   \nu^j_s\pi\\
& =  {n-N_s \choose n-i} \exp(-(n-i)\lambda_\infty\,(t-s))(1-\exp(-\lambda_\infty\,(t-s)))^{i-N_s}   \pi,
\end{align*}
from which the first assertion follows. 

Weak convergence can be proved in many ways. Let us first look at the case of one obligor, $n=1$. The integral in Equation~\eqref{eq:yalpha} is, with $\tau^{\alpha}=\tau^{1,\alpha}$ equal to 
\[
\frac{1}{\alpha}\int_0^{\alpha(\tau^\alpha\wedge t)}\lambda^\top X_u\,\dd u.
\]
Replacing the upper limit of the integral by $t$, this
 almost surely converges to $\lambda_\infty t$ for $\alpha\to\infty$. In fact this convergence is a.s.\ uniform.
Having already established the convergence in distribution of the $\tau^\alpha$, and by switching to an auxiliary space on which the $\tau^\alpha$ a.s.\ converge to $\tau^\infty$, we get 
\[
\frac{1}{\alpha}\int_0^{\alpha(\tau^\alpha\wedge t)}\lambda^\top X_u\,\dd u\to \int_0^{\alpha(\tau^\infty\wedge t)}\lambda^\top X_u\,\dd u.
\] 
This is sufficient, see \cite{kls} or \cite[Section~VIII.3d]{js} to conclude the weak convergence result for the case $n=1$. 

For the general case, one first notices that the process $N^\alpha$ is a sum of MM one point processes that are conditionally independent given $\cf^X$ and become independent in the limit. Combine this with the result for $n=1$. Alternatively, one could apply the results in \cite[Section~VII.3d]{js} again, although the computations will now be more involved.
\end{proof}

\subsection{Rapid switching for the MM Poisson process}

As before we replace $Q$ with $\alpha Q$ and let $\alpha\to\infty$ and denote $N^\alpha$ the corresponding counting process.
We apply Proposition~\ref{proposition:exp} to the matrix exponential $\exp\left(((e^{\ii u}-1)\diag(\lambda) + \alpha Q)(t-s)\right)$, and we find that the limit for $\alpha\to\infty$ equals $\exp((e^{\ii u}-1)\lambda_\infty(t-s))\pi\one^\top$. Hence, by virtue of \eqref{eq:phits}, we obtain $\ee[\exp(\ii uN^\alpha_t)X_t|\cf_s]$ $\to  \exp((e^{\ii u}-1)\lambda_\infty(t-s))\pi$ for the limit of the conditional characteristic function. This is just one of the many ways that eventually lead  to the conclusion that for $\alpha\to\infty$ the process $N^\alpha$ converge weakly to an ordinary Poisson process with constant intensity $\lambda_\infty$. In \cite{kls} one can find the stronger result that the variational distance between the MM law of $N^\alpha_t, t\in [0,T]$ and the limit law is of order $\alpha^{-1}$.

\end{document}